\documentclass{amsart}

\usepackage{amsthm}
\usepackage[leqno]{amsmath}
\usepackage{latexsym,amsfonts,amssymb}
\usepackage[all]{xy} \SelectTips{eu}{} \SilentMatrices
\usepackage{hyperref}

\newcommand{\numberseries}{\mdseries}   

\newlength{\thmtopspace}                
\newlength{\thmbotspace}                
\newlength{\thmheadspace}               
\newlength{\thmindent}                  

\renewcommand{\subparagraph}{\vspace*{\thmbotspace}}

\newtheoremstyle{bfupright head,slanted body}
                {\thmtopspace}{\thmbotspace}
                {\slshape}{\thmindent}{\bfseries}{.}{\thmheadspace}
                {{\numberseries \thmnumber{(#2) }}\thmnote{#3}}

\newtheoremstyle{bfupright head,upright body}
                {\thmtopspace}{\thmbotspace}
                {\upshape}{\thmindent}{\bfseries}{.}{\thmheadspace}
                {{\numberseries \thmnumber{(#2) }}\thmnote{#3}}

\newtheoremstyle{bfit head,upright body}
                {\thmtopspace}{\thmbotspace}
                {\upshape}{\thmindent}{\upshape}{.}{\thmheadspace}
                {{\numberseries\thmnumber{(#2) }}
                {\bfseries\itshape\thmnote{\negthickspace#3}}}

\newtheoremstyle{it head,upright body}
                {\thmtopspace}{\thmbotspace}
                {\upshape}{\thmindent}{\upshape}{.}{\thmheadspace}
                {{\numberseries\thmnumber{(#2) }}
                {\itshape\thmnote{\negthickspace#3}}}

\newtheoremstyle{fixed bf head,slanted body}
                {\thmtopspace}{\thmbotspace}{\slshape}
                {\thmindent}{\bfseries}{.}{\thmheadspace}
                {{\numberseries \thmnumber{(#2) }}\thmname{#1}\thmnote{ (#3)}}

\newtheoremstyle{fixed bf head,upright body}
                {\thmtopspace}{\thmbotspace}{\upshape}
                {\thmindent}{\bfseries}{.}{\thmheadspace}
                {{\numberseries \thmnumber{(#2) }}\thmname{#1}\thmnote{ (#3)}}

\newtheoremstyle{fixed bfit head,upright body}
                {\thmtopspace}{\thmbotspace}{\upshape}
                {\thmindent}{\bfseries\itshape}{.}{\thmheadspace}
                {{\numberseries \thmnumber{(#2) }}\thmname{#1}\thmnote{ (#3)}}

\newtheoremstyle{sc head,small body}
                {\thmtopspace}{\thmbotspace}
                {\small\upshape}{\thmindent}{\scshape}{.}{\thmheadspace}
                {\thmname{#1}}

\newtheoremstyle{numbered paragraph}
                {\thmtopspace}{\thmbotspace}{\upshape}
                {\thmindent}{\upshape}{}{0pt}
                {{\numberseries \thmnumber{(#2) }}}

\newtheoremstyle{unnumbered paragraph}
                {\thmtopspace}{\thmbotspace}{\upshape}
                {\parindent}{\upshape}{}{0pt}

\theoremstyle{bfupright head,slanted body}
\newtheorem{res}{}[section]             \newtheorem*{res*}{}

\theoremstyle{bfit head,upright body}
                 \newtheorem*{com*}{}

\theoremstyle{bfupright head,upright body}
\newtheorem{bfhpg}[res]{}               \newtheorem*{bfhpg*}{}

\theoremstyle{it head,upright body}
               \newtheorem*{ithpg*}{}

\theoremstyle{sc head,small body}

\theoremstyle{fixed bf head,slanted body}
\newtheorem{thm}[res]{Theorem}          \newtheorem*{thm*}{Theorem}
\newtheorem{prp}[res]{Proposition}      \newtheorem*{prp*}{Proposition}
\newtheorem{cor}[res]{Corollary}        \newtheorem*{cor*}{Corollary}
\newtheorem{lem}[res]{Lemma}            \newtheorem*{lem*}{Lemma}

\theoremstyle{fixed bf head,upright body}
\newtheorem{dfn}[res]{Definition}       \newtheorem*{dfn*}{Definition}
     \newtheorem*{con*}{Construction}
\newtheorem{obs}[res]{Observation}      \newtheorem*{obs*}{Observation}
\newtheorem{rmk}[res]{Remark}           \newtheorem*{rmk*}{Remark}
\newtheorem{exa}[res]{Example}          \newtheorem*{exa*}{Example}
         \newtheorem*{exe*}{Exercise}
            \newtheorem{stp*}{Setup}
         \newtheorem*{qst*}{Question}

\theoremstyle{numbered paragraph}

\theoremstyle{unnumbered paragraph}
\newtheorem{ipg*}{}

\newlength{\thmlistleft}        
\newlength{\thmlistright}       
\newlength{\thmlistpartopsep}   
\newlength{\thmlisttopsep}      
\newlength{\thmlistparsep}      
\newlength{\thmlistitemsep}     

\newcounter{eqc} 
\newenvironment{eqc}{\begin{list}{\upshape (\textit{\roman{eqc}})}%
    {\usecounter{eqc}%
      \setlength{\leftmargin}{\thmlistleft}%
      \setlength{\labelwidth}{\thmlistleft}%
      \setlength{\rightmargin}{\thmlistright}%
      \setlength{\partopsep}{\thmlistpartopsep}%
      \setlength{\topsep}{\thmlisttopsep}%
      \setlength{\parsep}{\thmlistparsep}%
      \setlength{\itemsep}{\thmlistitemsep}}}%
  {\end{list}}%

\newcounter{prt}
\newenvironment{prt}{\begin{list}{\upshape (\alph{prt})}%
    {\usecounter{prt}%
      \setlength{\leftmargin}{\thmlistleft}%
      \setlength{\labelwidth}{\thmlistleft}%
      \setlength{\rightmargin}{\thmlistright}%
      \setlength{\partopsep}{\thmlistpartopsep}%
      \setlength{\topsep}{\thmlisttopsep}%
      \setlength{\parsep}{\thmlistparsep}%
      \setlength{\itemsep}{\thmlistitemsep}}}%
  {\end{list}}%

\newcommand{\prtlbl}[1]{{\upshape(#1)}}

\newcounter{rqm}
\newenvironment{rqm}{\begin{list}{\upshape (\arabic{rqm})}%
    {\usecounter{rqm}%
      \setlength{\leftmargin}{\thmlistleft}%
      \setlength{\labelwidth}{\thmlistleft}%
      \setlength{\rightmargin}{\thmlistright}%
      \setlength{\partopsep}{\thmlistpartopsep}%
      \setlength{\topsep}{\thmlisttopsep}%
      \setlength{\parsep}{\thmlistparsep}%
      \setlength{\itemsep}{\thmlistitemsep}}}%
  {\end{list}}%

  {\end{list}}%

  {\end{list}}%

\newenvironment{prf*}[1][Proof]{%
  \begin{proof}[\bf #1]
    \setcounter{equation}{0}
    \renewcommand{\theequation}{\arabic{equation}}}
  {\end{proof}
}

  \newcommand{\proofoftag}[2][:]{(#2)#1}

\newcommand{\pgref}[1]{(\ref{#1})}

\newcommand{\thmref}[2][Theorem~]{#1\pgref{thm:#2}}
\newcommand{\corref}[2][Corollary~]{#1\pgref{cor:#2}}
\newcommand{\prpref}[2][Proposition~]{#1\pgref{prp:#2}}
\newcommand{\lemref}[2][Lemma~]{#1\pgref{lem:#2}}
\newcommand{\obsref}[2][Observation~]{#1\pgref{obs:#2}}

\newcommand{\dfnref}[2][Definition~]{#1\pgref{dfn:#2}}
\newcommand{\exaref}[2][Example~]{#1\pgref{exa:#2}}
\newcommand{\rmkref}[2][Remark~]{#1\pgref{rmk:#2}}

\newcommand{\secref}[2][Section~]{#1\ref{sec:#2}}

\renewcommand{\eqref}[1]{\pgref{eq:#1}}

\makeatletter
\def\@nobreak@#1{\mathchoice%
  {\nobreakdef@\displaystyle\f@size{#1}}%
  {\nobreakdef@\nobreakstyle\tf@size{\firstchoice@false #1}}%
  {\nobreakdef@\nobreakstyle\sf@size{\firstchoice@false #1}}%
  {\nobreakdef@\nobreakstyle\ssf@size{\firstchoice@false #1}}%
  \check@mathfonts}%
\def\nobreakdef@#1#2#3{\hbox{{%
                    \everymath{#1}%
                    \let\f@size#2\selectfont%
                    #3}}}%
\makeatother

\numberwithin{equation}{res}

\newcommand{\Spec}{\operatorname{Spec}}
\newcommand{\Supp}{\operatorname{Supp}}
\newcommand{\Ann}{\operatorname{Ann}}
\newcommand{\depth}{\operatorname{depth}}
\newcommand{\End}{\operatorname{End}}
\newcommand{\Hom}{\operatorname{Hom}}
\newcommand{\Ext}{\operatorname{Ext}}

\renewcommand{\Im}{\operatorname{Im}}
\newcommand{\Ker}{\operatorname{Ker}}
\newcommand{\Coker}{\operatorname{Coker}}
\newcommand{\M}{\operatorname{M}}
\newcommand{\D}{\mathsf{D}}
\newcommand{\RHom}{\mathbf{R}\!\operatorname{Hom}}

\def\widebardisplay#1{%
  \setbox0=\hbox{$\displaystyle #1$}
  \dimen0=\wd0%
  \advance\dimen0 by -4pt
  \vbox{%
    \nointerlineskip%
    \moveright 2pt 
    \vbox{\hrule width \dimen0}%
    \nointerlineskip%
    \kern 1pt
    \box0%
    }%
  }

\def\widebartext#1{%
  \setbox0=\hbox{$#1$}
  \dimen0=\wd0%
  \advance\dimen0 by -4pt
  \vbox{%
    \nointerlineskip%
    \moveright 2pt 
    \vbox{\hrule width \dimen0}%
    \nointerlineskip%
    \kern 1pt
    \box0%
    }%
  }

\def\widebarscript#1{%
  \setbox0=\hbox{$\scriptstyle #1$}
  \dimen0=\wd0%
  \advance\dimen0 by -3pt
  \vbox{%
    \nointerlineskip%
    \moveright 1.5pt 
    \vbox{\hrule width \dimen0}%
    \nointerlineskip%
    \kern .8pt
    \box0%
    }%
  }

\def\widebarscriptscript#1{%
  \setbox0=\hbox{$\scriptscriptstyle #1$}
  \dimen0=\wd0%
  \advance\dimen0 by -2pt
  \vbox{%
    \nointerlineskip%
    \moveright 1pt 
    \vbox{\hrule width \dimen0}%
    \nointerlineskip%
    \kern .6pt
    \box0%
    }%
  }

\begin{document}

\title[Construction of totally reflexive modules]{Construction of
  totally reflexive modules \\ from an exact pair of zero divisors}

\author{Henrik Holm \ } 
\address{Department of Basic Sciences and Environment, Faculty of Life
  Sciences, University of Copenhagen, Thorvaldsensvej 40, DK-1871
  Frederiksberg C, Denmark}
\email{hholm@life.ku.dk} 
\urladdr{http://www.matdat.life.ku.dk/\~{}hholm/}






\begin{abstract}
  Let $A$ be a local ring which admits an exact pair $x,y$ of zero
  divisors as defined by Henriques and {\c{S}}ega. Assuming that this
  pair is regular and that there exists a regular element on the
  $A$-module $A/(x,y)$, we explicitly construct an infinite family of
  non-isomorphic indecomposable totally reflexive $A$-modules. In this
  setting, our construction provides an answer to a question raised by
  Christensen, Piepmeyer, Striuli, and Takahashi.  Furthermore, we
  compute the module of homomorphisms between any two given modules
  from the infinite family mentioned above.
\end{abstract}

\maketitle

\section{Introduction}

\noindent
Throughout this paper, $A$ is a commutative noetherian local ring.

As indicated by the title, this paper is concerned with explicit
constructions of totally reflexive $A$-modules, as defined by
Auslander~\cite{MAs67} in 1967 (the terminology ``totally reflexive''
was introduced in 2002 by Avramov and Martsinkovsky \cite{LLAAMr02}).

Totally reflexive $A$-modules always exist, indeed, $A$ itself is
totally reflexive. However, in general, there need not exist non-free
totally reflexive $A$-modules. If $A$ is Gorenstein then the totally
reflexive $A$-modules are exactly the maximal Cohen-Macaulay modules,
and their representation theory is a classical field of study. A main
result of Christensen, Piepmeyer, Striuli, and Takahashi \cite[Theorem
B]{CPST-08} asserts that if $A$ is not Gorenstein, then existence of
one non-free totally reflexive $A$-module implies the existence of
infinitely many non-isomorphic indecomposable totally reflexive
$A$-modules. Unfortunately, the known proof of this interesting result
is not constructive, which is why the authors of \cite{CPST-08} raise
the following.

\begin{qst*}[{\cite[(4.8)]{CPST-08}}]
  Assume that $A$ is not Gorenstein, and that there exists a non-free
  totally reflexive $A$-module. Are there constructions that produce
  infinite families of non-isomorphic indecomposable totally reflexive
  $A$-modules?
\end{qst*}

Under suitable assumptions, we give in this paper exactly such a
construction. However, before we reveal the details of our
construction, we will mention a couple of related results from the
literature.

Assume that $A$ is complete or has an uncountable residue field.
Assume furthermore that there exist a prime ideal $\mathfrak{p}$ in
$A$ with $\operatorname{grade}\mathfrak{p}>0$ and $\dim
A/\mathfrak{p}>1$, and a totally reflexive $A$-module $M$ such that
$M_{\mathfrak{p}}$ is not $A_{\mathfrak{p}}$-free. Then Takahashi
\cite{RTk07a} proves the existence of uncountably many non-isomorphic
indecomposable totally reflexive $A$-modules. The proof is not
constructive, but in \cite[Example 4.3]{RTk07a} one does findi an example
(\mbox{$A=Q[\hspace{-1.4pt}[x,y,z]\hspace{-1.4pt}]/(x^2)$}, where $Q$
is a complete local domain which is not a field) where uncountably
many non-isomorphic indecomposable totally reflexive $A$-modules are
actually constructed.

Assume that $A$ has an embedded deformation of codimension \mbox{$c
  \geqslant 2$}, i.e.~there~is a local ring $(Q,\mathfrak{q},k)$ and a
$Q$-regular sequence $\boldsymbol{x}$ in $\mathfrak{q}^2$ of length
$c$ such that \mbox{$A\cong Q/(\boldsymbol{x})$}. Avramov, Gasharov,
and Peeva \cite{AGP-97} constructs a non-free totally reflexive
$A$-module $G$, from which infinitely many non-isomorphic
indecomposable totally reflexive $A$-modules can be constructed by
using results of Avramov and Iyengar \cite{LLASBI07}:

Consider the graded ring \mbox{$\mathcal{R}=k[\chi_1,\ldots,\chi_c]$}
where $\chi_i$ has degree $2$.  From \cite[4.1]{LLASBI07} and
\cite[proof of 7.4(1)]{LLASBI07} it follows that given a closed subset
\mbox{$X \subseteq \Spec \mathcal{R}$} there is a module \mbox{$M_X
  \in \mathsf{Thick}_A(G\oplus A) \subseteq \mathsf{D}(A)$} with
cohomological support $\Supp^*(M_X,k)$ equal to $X$.  As \mbox{$M_X
  \in \mathsf{Thick}_A(G\oplus A)$} some syzygy $G_X$ of $M_X$ is
totally reflexive, and by \cite[4.2]{LLASBI07} one also has
\mbox{$\Supp^*(G_X,k) = X$}. Clearly, if \mbox{$X \neq Y$} then $G_X$
and~$G_Y$ are not isomorphic.  Furthermore, if $X$ is irreducible, it
follows by \cite[(3.6.2)]{LLASBI07} that $\Supp^*(\tilde{G}_X,k) = X$
for some indecomposable summand $\tilde{G}_X$ of $G_X$.

In the author's opinion, the construction of $M_X$ given in
\cite[proof of 4.1]{LLASBI07} is not as explicit and accessible as one
could hope for. Under suitable assumptions, this paper offers a
construction of infinitely many non-isomorphic indecomposable totally
reflexive $A$-modules, which is different from the ones mentioned
above.  Our construction has the advantage of being fairly elementary;
it applies to large classes of examples, cf.~ \exaref{fg} and
\lemref{depth2}; and it makes explicit computations of relevant
$\Hom$-modules possible.

Given an exact pair of zero divisors $x,y$ in $A$, as defined by
Henriques and {\c{S}}ega \cite{HS}, it is easily seen that $A/(x)$ and
$A/(y)$ are non-free totally reflexive $A$-modules.  Thus, if $A$ is
not Gorenstein, one should by the main result in \cite{CPST-08} be
able to construct infinitely many non-isomorphic indecomposable
totally reflexive $A$-modules. In this paper, we define for each $a$
in $A$ two totally reflexive $A$-modules $G_a$ and $H_a$ as cokernels
of certain $2\!\times\!2$ matrices with entries in $A$. A special case
of our main \thmref{main} reads as follows.

\begin{thm*}
  Let \mbox{$x,y \in A$} be a regular exact pair of zero
  divisors, and let $a \in A$. If $a$ is regular on the $A$-module
  $A/(x,y)$ then
  \begin{displaymath}
    G_{a}, G_{a^2}, G_{a^3}, G_{a^4}, \ldots, 
    H_{a}, H_{a^2}, H_{a^3}, H_{a^4}, \ldots
  \end{displaymath}
  is a (double) infinite family of non-isomorphic, indecomposable,
  non-free, totally reflexive $A$-modules.  Furthermore, one has the
  following identities,
  \begin{align*}
    \Hom_A(H_{a_m},G_{a^n}) \cong \Hom_A(H_{a^n},G_{a^m}) &\cong\, G_{a^{m+n}},
    \\
    \Hom_A(G_{a^n},H_{a^m}) \cong \Hom_A(G_{a^m},H_{a^n}) &\cong\, H_{a^{m+n}},
    \\
    \Hom_A(G_{a^m},G_{a^n}) \cong \Hom_A(H_{a^n},H_{a^m}) &\cong 
    \left\{\!\!
      \begin{array}{ccl}
        H_{a^{m-n}} & \text{for} & m>n \\
        A & \text{for} & m=n \\
        G_{a^{n-m}} & \text{for} & m<n
      \end{array}
    \right.\!\!.
  \end{align*}
\end{thm*}

The paper is organized as follows: In \secref{pair}, we introduce
exact pairs of zero divisors and discuss regularity. In \secref{two
  families}, the modules $G_a$ and $H_a$ are defined, and in
\secref[Sections~]{ind} and \secref[]{iso} we study the
indecomposability and the isomorphism classes of these modules.
\secref{main} simply contains our main result.  The final \secref{Hom}
is the tecnical part of this paper. In this section, on which several
results in \secref[Sections~]{ind}, \secref[]{iso}, and
\secref[]{main} are based, we demonstrate how to compute the module of
homomorphisms between various combinations of $G_a$ and $H_b$.

\section{Exact pairs of zero divisors and regular elements} \label{sec:pair}

In this section, we consider an exact pair of zero divisors as defined
by Henriques and {\c{S}}ega \cite{HS}. We introduce a notion of
regularity for such a pair and give examples.

\begin{dfn}
  \label{dfn:xy}
  Two non-units $x,y \in A$ are called an \emph{exact pair of zero
    divisors} if $\Ann_A(x)=(y)$ and $\Ann_A(y)=(x)$.
\end{dfn}

Let $M$ be a finitely generated $A$-module, and let \mbox{$a \in A$}
be an element.  Recall that $a$ is \emph{weakly regular on $M$} if
multiplication by $a$ on $M$ is a monomorphism. If, in addition, the
element $a$ is not a unit, then $a$ is called \emph{regular on $M$}.

A (weakly) regular element on the $A$-module \mbox{$M=A$} is simply
referred to as a (weakly) regular element in the ring $A$.

\begin{lem}
  \label{lem:regular}
  Let $x,y \in A$ be an exact pair of zero divisors. Then
  the following conditions are equivalent:
  \begin{eqc}
  \item $x$ is regular on $A/(y)$.
  \item $y$ is regular on $A/(x)$.
  \item $(x) \cap (y) = 0$.
  \end{eqc}
\end{lem}

\begin{proof}
  By symmetry, it suffices to prove $(i) \Leftrightarrow (iii)$. 

  First assume $(i)$, and let $z$ be in \mbox{$(x) \cap (y)$}. As
  \mbox{$z \in (x)$} we have \mbox{$z=ax$} for some $a$. As $z \in
  (y)$ it follows that \mbox{$x[a]_{(y)} = [z]_{(y)} = [0]_{(y)}$} in
  $A/(y)$, and since multiplication by $x$ on $A/(y)$ is a
  monomorphism, we conclude that $[a]_{(y)}=[0]_{(y)}$, that is, $a
  \in (y)$.  Consequently, $z=ax=0$ since, in particular, $yx=0$.

  Next assume $(iii)$. We must argue that multiplication by $x$ on
  $A/(y)$ is a mono\-mor\-phism. Thus, let $[a]_{(y)}$ in $A/(y)$
  satisfy \mbox{$x[a]_{(y)} = [0]_{(y)}$}, that is, $ax \in (y)$.
  Then \mbox{$ax \in (x) \cap (y)=0$}, and hence $a \in
  \Ann_A(x)=(y)$. Thus $[a]_{(y)}=[0]_{(y)}$ in $A/(y)$.
\end{proof}

\begin{dfn}
  \label{dfn:regular}
  An exact pair of zero divisors \mbox{$x,y \in A$} that
  satisfies the equivalent conditions in \lemref{regular} is called a
  \emph{regular} exact pair of zero divisors.
\end{dfn}

\begin{exa}
  \label{exa:fg}
  Let $Q$ be a commutative noetherian local ring, and let $f, g$ be
  regular elements in $Q$. Set \mbox{$A=Q/(fg)$} and let $x,y \in A$
  denote the cosets of $f,g$ with respect to the ideal $(fg)$.  Then
  $x,y$ is an exact pair of zero divisors in $A$; and this pair is
  regular if and only if $(f) \cap (g) = (fg)$ in $Q$.

  In particular, if $f$ and $g$ are non-zero and non-units in a
  noetherian local UFD $Q$, then $x,y$ is an exact pair of zero
  divisors in $A$; and this pair is regular if and only if $f$ and $g$
  are relatively prime. 

  Note that local UFDs come in all shapes and sizes (Gorenstein,
  non-Gorenstein etc.), in fact, by Heitmann \cite[Theorem 8]{RCH93}
  every complete local ring of depth \mbox{$>\!1$}, in which no
  integer is a zero divisor, is the completion of a local UFD.
\end{exa}

\begin{rmk}
  \label{rmk:depth}
  Let $\mathfrak{a}$ be a proper ideal in $A$, and let $M$ be a
  finitely generated module over $A/\mathfrak{a}$. Then
  \mbox{$\depth_A\!M = \depth_{A/\mathfrak{a}}\!M$} as a sequence
  $x_1,\ldots,x_n$ of non-units in $A$ is regular on $M$ if and only
  if the corresponding sequence
  $[x_1]_{\mathfrak{a}},\ldots,[x_n]_{\mathfrak{a}}$ in
  $A/\mathfrak{a}$ is regular on $M$.  In particular, the depth of the
  $A$-module $A/\mathfrak{a}$ equals the depth of the ring
  $A/\mathfrak{a}$.
\end{rmk}

Several results in this paper refer to a weakly regular element $a$ on
the $A$-module $A/(x,y)$. For example, $a$ could be a unit in $A$.
However, only in the case where $a$ is regular on the $A$-module
$A/(x,y)$ do we get interesting applications of our results. For the
existence of regular elements on $A/(x,y)$, we give the following.

\begin{lem}
  \label{lem:depth2}
  Let $Q$ be a commutative noetherian local ring, and let $f, g$ be
  regular elements in $Q$ such that \mbox{$(f) \cap (g) = (fg)$}.
  Consider the regular exact pair $x,y$ of zero divisors in
  \mbox{$A=Q/(fg)$} constructed in \exaref{fg}.

  If $\depth Q>2$ then there exists regular elements on the $A$-module
  $A/(x,y)$.
\end{lem}

\begin{proof}
  We must prove that the $A$-module $A/(x,y)$ has \mbox{$\depth>0$},
  equivalently, that the ring $A/(x,y)$ has \mbox{$\depth>0$},
  cf.~\rmkref{depth}. Since there is a ring isomorphism \mbox{$A/(x,y)
    \cong Q/(f,g)$}, we are required to show that $Q/(f,g)$ has
  \mbox{$\depth>0$}.

  To finish the proof it suffices to argue that $f,g$ is a $Q$-regular
  sequence, because then $\depth Q/(f,g) = \depth Q -2$ by Bruns and
  Herzog \cite[Proposition 1.2.10(d)]{BruHer}. However, it is given
  that $f$ is regular on $Q$, and that $g$ is regular on $Q/(f)$
  follows from the assumption \mbox{$(f) \cap (g) = (fg)$} and the
  regularity of $g$ on $Q$.
\end{proof}





\section{Two families of totally reflexive modules} \label{sec:two families}

Assuming that the ring $A$ admits an exact pair of zero divisors, we
introduce in this section two families $(G_a)_{a \in A}$ and $(H_a)_{a
  \in A}$ of totally reflexive $A$-modules.

We begin with the following definition due to Auslander
\cite[\S3.2.2]{MAs67}.  Although Auslander himself did not use the
terminology ``totally reflexive module'', this usage was introduced by
Avramov and Martsinkovsky \cite[\S2]{LLAAMr02} and has grown to be
standard in several later papers on the subject.

\begin{dfn}
  \label{dfn:TR}
  A finitely generated $A$-module $G$ is \emph{totally reflexive} if
  it satisfies the following three conditions:
  \begin{rqm}
  \item $\Ext^i_A(G,A)=0$ for all $i>0$.
  \item $\Ext^i_A(\Hom_A(G,A),A)=0$ for all $i>0$.
  \item The biduality map $G \longrightarrow \Hom_A(\Hom_A(G,A),A)$ is
    an isomorphism.
  \end{rqm}
\end{dfn}

\begin{dfn}
  \label{dfn:GH}
  Let \mbox{$x,y \in A$} be an exact pair of zero divisors,
  and let $a \in A$. We define two $2\!\times\!2$ matrices,
  \begin{displaymath}
    \gamma_a = 
      \begin{pmatrix}
        x & a \\
        0 & y
      \end{pmatrix}
    \quad \text{ and } \quad
    \eta_a = 
      \begin{pmatrix}
        y & -a \\
        0 & x
      \end{pmatrix}.
  \end{displaymath}
  Furthermore, considering $\gamma_a$ and $\eta_a$ as $A$-linear maps
  \mbox{$A^2 \to A^2$} acting on column vectors by multiplication from
  the left, we define two finitely generated $A$-modules,
  \begin{displaymath}
    G_a = \Coker \gamma_a
    \quad \text{ and } \quad
    H_a = \Coker \eta_a.
  \end{displaymath}
\end{dfn}

\begin{obs}
  \label{obs:unit}
  Let $x,y \in A$ be an exact pair of zero divisors, let
  $a$ be any element in $A$, and let $u$ be a unit in $A$. Then the
  commutative diagrams,
  \begin{displaymath}
    \begin{gathered}
    \xymatrix@R=6ex@C=8ex{
      A^2 \ar[r]^-{\gamma_{ua}}
    \ar[d]^-{\cong}_-{\text{\small
    $\left(\!\!
    \begin{array}{cc}
      1 & 0 \\
      0 & u
    \end{array}
    \!\!\right)$}} & A^2 
    \ar[d]_-{\cong}^-{\text{\small
    $\left(\!\!
    \begin{array}{cc}
      1 & 0 \\
      0 & u
    \end{array}
    \!\!\right)$}} \\
      A^2 \ar[r]_-{\gamma_a} & A^2
    }
    \end{gathered}
    \qquad \text{and} \qquad
    \begin{gathered}
    \xymatrix@R=6ex@C=8ex{
      A^2 \ar[r]^-{\eta_{ua}}
    \ar[d]^-{\cong}_-{\text{\small
    $\left(\!\!
    \begin{array}{cc}
      1 & 0 \\
      0 & u
    \end{array}
    \!\!\right)$}} & A^2 
    \ar[d]_-{\cong}^-{\text{\small
    $\left(\!\!
    \begin{array}{cc}
      1 & 0 \\
      0 & u
    \end{array}
    \!\!\right)$}} \\
      A^2 \ar[r]_-{\eta_a} & A^2
    }
    \end{gathered}
  \end{displaymath}
  show that $G_{ua} \cong G_a$ and $H_{ua} \cong H_a$. In particular,
  $G_{-a} \cong G_a$ and $H_{-a} \cong H_a$.
\end{obs}

\begin{lem}
  \label{lem:gamma}
  Let $x,y \in A$ be an exact pair of zero divisors, and
  let $a \in A$.  Then there is an exact complex of free $A$-modules
  given by
  \begin{displaymath}
    \boldsymbol{F} = \ 
    \cdots \longrightarrow A^2 
    \stackrel{\gamma_a}{\longrightarrow} A^2 
    \stackrel{\eta_a}{\longrightarrow} A^2 
    \stackrel{\gamma_a}{\longrightarrow} A^2 
    \stackrel{\eta_a}{\longrightarrow} \cdots.
  \end{displaymath}
  Furthermore, there is an isomorphism of complexes
  $\Hom_A(\boldsymbol{F},A) \stackrel{\cong}{\longrightarrow}
  \boldsymbol{F}$,
  \begin{displaymath}
    \xymatrix{
      \cdots \ar[r]^-{\eta_a^t} &
      A^2 \ar[r]^-{\gamma_a^t} 
      \ar[d]^-{\varphi}_-{\cong} & A^2 \ar[r]^-{\eta_a^t} 
      \ar[d]^-{\varphi}_-{\cong} & A^2
      \ar[d]^-{\varphi}_-{\cong} \ar[r]^-{\gamma_a^t} & \cdots
      \\
      \cdots \ar[r]^-{\gamma_a} & A^2
      \ar[r]^-{\eta_a} & A^2 \ar[r]^-{\gamma_a} & A^2 \ar[r]^-{\eta_a}
      & \cdots}
  \end{displaymath}
  where $(-)^t$ is transposition of matrices, and $\varphi$ is the
  isomorphism given by
  \begin{displaymath}
    \varphi = 
      \begin{pmatrix}
        0 & 1 \\
        -1 & 0
      \end{pmatrix}.
  \end{displaymath}
  In particular, the complex $\Hom_A(\boldsymbol{F},A)$ is exact.
\end{lem}

\begin{proof}
  As \mbox{$\gamma_a\eta_a=\eta_a\gamma_a=0$} we conclude that
  $\boldsymbol{F}$ is a complex. To show that $\boldsymbol{F}$ is
  exact, we must argue that \mbox{$\Ker \gamma_a \subseteq \Im
    \eta_a$} and \mbox{$\Ker \eta_a \subseteq \Im \gamma_a$}.  We only
  prove the first inclusion, as the other one is proved analogously.
  To this end, assume that
  \begin{displaymath}
    \gamma_a
    \begin{pmatrix}
      b_1 \\ b_2
    \end{pmatrix}
    =
    \begin{pmatrix}
      b_1x+b_2a \\ b_2y
    \end{pmatrix}
    =
    \begin{pmatrix}
      0 \\ 0
    \end{pmatrix}.
  \end{displaymath}
  As $b_2y=0$ we get that $b_2=c_2x$ for some $c_2$ in $A$. Hence
  $(b_1+c_2a)x = b_1x+b_2a = 0$, so there exists $c_1$ in $A$ with
  $b_1+c_2a=c_1y$. Consequently,
  \begin{displaymath}
    \eta_a
    \begin{pmatrix}
      c_1 \\ c_2
    \end{pmatrix}
    =
    \begin{pmatrix}
      c_1y-c_2a \\ c_2x
    \end{pmatrix}
    =
    \begin{pmatrix}
      b_1 \\ b_2
    \end{pmatrix},
  \end{displaymath}
  as desired.  The last assertion of the lemma is straightforward to
  check.
\end{proof}

\begin{prp}
  \label{prp:gamma}
  Let $x,y \in A$ be an exact pair of zero divisors, and
  let $a \in A$.  Then the $A$-modules $G_a$ and $H_a$ are totally
  reflexive. Furthermore, $G_a$ and $H_a$ are each others dual, that
  is, $G_a \cong \Hom_A(H_a,A)$ and $H_a \cong \Hom_A(G_a,A)$.
\end{prp}

\begin{proof}
  By \lemref{gamma}, the exact complex $\boldsymbol{F}$ consists of
  finitely generated free $A$-modules, and furthermore
  $\Hom_A(\boldsymbol{F},A)$ is exact. Since $G_a$ and $H_a$ are
  cokernels of differentials in $\boldsymbol{F}$, it follows that they
  are totally reflexive, cf.~\cite[prop.~8 in \S3.2.2]{MAs67}.

  The duality is immediate from the isomorphism
  $\boldsymbol{F} \cong \Hom_A(\boldsymbol{F},A)$.
\end{proof}

The proof of the following result uses the derived category $\D(A)$,
and the right derived $\Hom$-functor $\RHom_A(-,-)$. We refer to
Weibel \cite[Chapter 10]{Wei} for details.

\begin{cor}
  \label{cor:Ext}
  Let $x,y \in A$ be an exact pair of zero divisors, and
  let $a,b \in A$.  Then there are the following isomorphisms of
  $A$-modules:
  \begin{prt}
  \item $\Ext_A^i(H_b,G_a) \cong \Ext_A^i(H_a,G_b)$ for
    every $i \in \mathbb{Z}$.
  \item $\Ext_A^i(G_a,H_b) \cong \Ext_A^i(G_b,H_a)$ for
    every $i \in \mathbb{Z}$.
  \item $\Ext_A^i(G_a,G_b) \cong \Ext_A^i(H_b,H_a)$ for
    every $i \in \mathbb{Z}$.
  \end{prt}
\end{cor}

\begin{proof}
  We only prove part (a), as the proofs of (b) and (c) are similar.

  As $H_a$ is totally reflexive, \mbox{$\Ext_A^i(H_a,A)=0$} for all
  $i>0$. By \prpref{gamma} we also have $G_a \cong \Hom_A(H_a,A)$, and
  therefore $G_a \cong \RHom_A(H_a,A)$ in the derived category
  $\D(A)$. Consequently, one has the following isomorphisms in $\D(A)$,
  \begin{align*}
    \RHom_A(H_b,G_a) &\cong \RHom_A(H_b,\RHom_A(H_a,A)) \\
    &\cong \RHom_A(H_a,\RHom_A(H_b,A)) \\
    &\cong \RHom_A(H_a,G_b),
  \end{align*}
  where the second isomorphism is the so-called swap-isomorphism. The
  desired result is obtained by taking homology on the displayed
  isomorphisms.
\end{proof}

\begin{rmk}
  \corref{Ext}(c) implies that $\Hom_A(G_a,G_a)$ and $\Hom_A(H_a,H_a)$
  are isomorphic as $A$-modules. One should be careful to conclude
  from this fact alone that $\End_A(G_a)$ and $\End_A(H_a)$ are
  isomorphic rings, cf.~\prpref{End-iso}.
\end{rmk}


\begin{prp}
  \label{prp:exact}
  Let $x,y \in A$ be an exact pair of zero divisors, and
  let $a \in A$. 
  \begin{prt}
  \item If $a$ is weakly  regular on $A/(y)$  then $G_a$ is isomorphic
    to the ideal $(y,a)$.
  \item If $a$ is weakly  regular on $A/(x)$  then $H_a$ is isomorphic
    to the ideal $(x,a)$.
  \end{prt}
\end{prp}
 
\begin{proof}
  We only prove part \prtlbl{a} as the proof of \prtlbl{b} is similar.
  Since $G_a = \Coker \gamma_a$, the assertion will follow once we
  have established exactness of the sequence,
  \begin{displaymath}
    \xymatrix{
      A^2 \ar[r]^-{\gamma_a} & A^2 \ar[r]^-{(y \ -a)} & (y,a) \ar[r] & 0.
    } 
  \end{displaymath}  
  Clearly, the image of \mbox{$(y \ -\!a)$} is all of $(y,a)$, and its
  composite with $\gamma_a$ is zero:
  \begin{displaymath}
    (y \ -\!a)\gamma_a = 
    (y \ -\!a)
      \begin{pmatrix}
        x & a \\
        0 & y
      \end{pmatrix}
     = (0 \ \ 0).
  \end{displaymath}
  It remains to see that $\Ker(y \ -\!a) \subseteq \Im \gamma_a$.
  Thus, assume that $(b_1,b_2) \in A^2$ satisfies
  \begin{displaymath}
    b_1y-b_2a = (y \ -\!a)
      \begin{pmatrix}
        b_1 \\
        b_2
      \end{pmatrix}
     = 0.
  \end{displaymath}
  As \mbox{$b_2a = b_1y \in (y)$} and since $a$ is weakly regular on
  $A/(y)$, it follows that \mbox{$b_2 \in (y)$},
  i.e.~\mbox{$b_2=c_2y$} for some $c_2$. Now
  \mbox{$(b_1-c_2a)y=b_1y-b_2a=0$}, and as \mbox{$\Ann_A(y)=(x)$} we
  conclude that $b_1-c_2a=c_1x$ for some $c_1$. Consequently,
  \begin{displaymath}
      \gamma_a
      \begin{pmatrix}
        c_1 \\
        c_2
      \end{pmatrix}
     = 
      \begin{pmatrix}
        x & a \\
        0 & y
      \end{pmatrix}
      \begin{pmatrix}
        c_1 \\
        c_2
      \end{pmatrix}
     = 
      \begin{pmatrix}
        c_1x+c_2a \\
        c_2y
      \end{pmatrix}
     =
      \begin{pmatrix}
        b_1 \\
        b_2
      \end{pmatrix},
  \end{displaymath} 
  that is, $(b_1,b_2) \in \Im \gamma_a$, as desired. 
\end{proof}

\begin{cor}
  \label{cor:free--non-free}
  Let $x,y \in A$ be an exact pair of zero divisors, and
  let $a \in A$. 
  \begin{prt}
  \item If $a$ is a unit then $G_a \cong H_a \cong A$.
  \item If $a$ is not a unit then neither $G_a$ nor $H_a$ is free.
  \end{prt}
\end{cor}


\begin{proof}
  \proofoftag{a} Immediate from \prpref{exact}.

  \proofoftag{b} If $a$ is not a unit then all the entries in
  $\gamma_a$ and $\eta_a$ belong to $A$.  Consequently, the
  homomorphisms $k \otimes_A \gamma_a$ and $k \otimes_A \eta_a$ are
  zero, and it follows that
  \begin{displaymath}
    \cdots \longrightarrow A^2 
    \stackrel{\gamma_a}{\longrightarrow} A^2 
    \stackrel{\eta_a}{\longrightarrow} A^2 
    \stackrel{\gamma_a}{\longrightarrow} A^2 
    \longrightarrow G_a \longrightarrow 0
  \end{displaymath}
  is an augmented minimal free resolution of $G_a$. Hence $G_a$ has
  infinite projective dimension. Analogously one sees that $H_a$ has
  infinite projective dimension.
\end{proof}

\section{Indecomposability} \label{sec:ind}

For a general ring element $a$, the $A$-modules $G_a$ and $H_a$ are
\emph{not} indecomposable, cf.~\exaref{decomposable}. However, under
suitable assumptions we will prove that $\End_A(G_a)$ and
$\End_A(H_a)$ are both isomorphic to $A$, in particular, the
endomorphism rings are (commutative, noetherian and) local. Hence
$G_a$ and $H_a$ are indecomposable in a quite strong sense.




We begin by establishing a relation between the endomorphism rings of
$G_a$ and $H_a$ which is valid for any ring element $a$.

\begin{obs}
  Let $F$ be an $A$-linear contravariant endofunctor on the category
  of $A$-modules, and let $M$ be an $A$-module. Then there is an
  induced homomorphism of $A$-algebras, $\End_A(M)^{\mathrm{op}} \to
  \End_A(FM)$ given by $\alpha \mapsto F\alpha$.
\end{obs}

\begin{prp}
  \label{prp:End-iso}
  Let $M$ be a reflexive $A$-module. Then the induced homomorphism
  \mbox{$\varepsilon_M \colon \End_A(M)^{\mathrm{op}} \to
    \End_A(\Hom_A(M,A))$} of $A$-algebras is an isomorphism.  In
  particular, $M$ is indecomposable if and only if $\,\Hom_A(M,A)$ is
  indecomposable.
\end{prp}

\begin{proof}
  For every $A$-module $M$ there is a commutative diagram of $A$-modules,
  \begin{displaymath}
    \xymatrix@C=-7ex@R=7ex{
      \Hom(M,M) \ar[dr]_-{\delta_M \circ -} \ar[rr]^-{\varepsilon_M} & & 
      \Hom(\Hom(M,A),\Hom(M,A)) 
      \\
      & \Hom(M,\Hom(\Hom(M,A),A)) \ar[ur]_-{\Sigma}^-{\cong} &
    }
  \end{displaymath}
  where $\delta_M \colon M \to \Hom(\Hom(M,A),A)$ is the
  biduality homomorphism, and $\Sigma$ is the so-called swap
  isomorphism given by $\Sigma(\theta)(\omega)(m) = \theta(m)(\omega)$
  for $\omega \in \Hom(M,A)$ and $\theta \in
  \Hom(M,\Hom(\Hom(M,A),A))$ and $m \in M$.  By definition, $M$ is
  reflexive if $\delta_M$ is an isomorphism, in which case the map
  \mbox{$\delta_M\!\circ\!-$} is an isomorphism, and the commutative
  diagram above establishes the desired $A$-algebra isomorphism.

  The last assertion in the proposition follows from the already
  established ring isomorphism, as an $A$-module is indecomposable if
  and only if its endomorphism ring has no non-trivial idempotents.
\end{proof}

\begin{cor}
  \label{cor:EndG-EndH}
  Let \mbox{$x,y \in A$} be an exact pair of zero divisors,
  and let \mbox{$a \in A$}. Then there is an isomorphism
  $\End_A(G_a)^{\mathrm{op}} \cong \End_A(H_a)$ of $A$-algebras. In
  particular, $G_a$ is indecomposable if and only if $H_a$ is
  indecomposable.
\end{cor}

\begin{proof}
  Apply \prpref{End-iso} to the (totally) reflexive module
  \mbox{$M=G_a$}, and use the fact that \mbox{$\Hom_A(G_a,A) \cong
    H_a$} by \prpref{gamma}.
\end{proof}

The following result is probably folklore. However, since the author
was not able to find a reference, a proof has been included.

\begin{lem}
  \label{lem:Foxby}
  Let $M$ be a finitely generated $A$-module. If $A$ and $\Hom_A(M,M)$
  are isomorphic as $A$-modules, then the canonical homomorphism of
  $A$-algebras, given by $\chi \colon A \to \End_A(M)$, $\chi(a)=a1_M$,
  is also an isomorphism.
\end{lem}

\begin{proof}
  Let \mbox{$\zeta \colon A \to \Hom_A(M,M)$} be an isomorphism of
  $A$-modules. As $\zeta$ is surjective, there exists $b$ in $A$ such
  that \mbox{$b\zeta(1)=\zeta(b)=1_M$}. It follows that \mbox{$bM=M$}.
  Furthermore, $M \neq 0$ since $\Hom_A(M,M)$ is non-zero. Hence
  Nakayama's Lemma \cite[Theorem 2.2]{Mat} implies that $b$ is a unit
  in $A$, and thus $b\zeta$ is also an isomorphism.  It remains to
  note that $b\zeta=\chi$ since
  $b\zeta(a)=ba\zeta(1)=a\zeta(b)=a1_M=\chi(a)$.
\end{proof}

\begin{thm}
  \label{thm:End-A}
  Let $x,y \in A$ be a regular exact pair of zero divisors,
  and let $a \in A$ be weakly regular on $A/(x,y)$.  Then there are
  isomorphisms of $A$-algebras,
  \begin{displaymath}
    \End_A(G_a) \cong A \cong \End_A(H_a).
  \end{displaymath}
  In particular, the endomorphism rings $\End_A(G_a)$ and
  $\End_A(H_a)$ are (commutative, noetherian and) local, and thus the
  $A$-modules $G_a$ and $H_a$ are indecomposable.
\end{thm}

\begin{proof}
  Taking \mbox{$b=1$} in \thmref{Hom-G-ab-a}(a), we get the first two
  isomorphism of $A$-modules in the chain \mbox{$\Hom_A(H_a,H_a) \cong
    \Hom_A(G_a,G_a) \cong H_1 \cong A$}. The last isomorphism is by
  \corref{free--non-free}(a). Now \lemref{Foxby} completes the proof.
\end{proof}

\begin{exa}
  \label{exa:decomposable}
  For a general ring element $a$, the conclusion in \thmref{End-A}
  fails.  For example, if $a$ belongs to $(x)$, say $a=qx$, then the
  commutative diagram,
  \begin{displaymath}
    \xymatrix@R=7ex@C=15ex{
      A^2 \ar[r]^-{\text{\small
    $\left(\!\!
    \begin{array}{cc}
      x & a \\
      0 & y
    \end{array}
    \!\!\right)$}}
    \ar[d]^-{\cong}_-{\text{\small
    $\left(\!\!
    \begin{array}{cc}
      1 & q \\
      0 & 1
    \end{array}
    \!\!\right)$}} & A^2 
    \ar[d]_-{\cong}^-{\text{\small
    $\left(\!\!
    \begin{array}{cc}
      1 & 0 \\
      0 & 1
    \end{array}
    \!\!\right)$}} \\
      A^2 \ar[r]_-{\text{\small
    $\left(\!\!
    \begin{array}{cc}
      x & 0 \\
      0 & y
    \end{array}
    \!\!\right)$}} & A^2
    }
  \end{displaymath}
  shows that
  \begin{displaymath}
    G_a = \Coker \gamma_a = 
    \Coker
    \begin{pmatrix}
      x & a \\
      0 & y
    \end{pmatrix} \cong
    \Coker
    \begin{pmatrix}
      x & 0 \\
      0 & y
    \end{pmatrix} \cong
    A/(x) \oplus A/(y),
  \end{displaymath}
  which is not indecomposable.
\end{exa}

\section{Isomorphism classes} \label{sec:iso}

In this section, we investigate whether two given modules from the
joint family \mbox{$(G_a)_{a \in A} \cup (H_a)_{a \in A}$} are
isomorphic. Of course, they might very well be exactly that since for
example \mbox{$G_1 \cong A \cong H_1$} by \corref{free--non-free}(a),
and furthermore $G_a \cong G_{ua}$ whenever $u$ is a unit, see
\obsref{unit}.

\begin{thm}
  \label{thm:iso-GH}
  Let $x,y \in A$ be a regular exact pair of zero divisors.
  Furthermore, let $a,b$ be elements in $A$ such that:
  \begin{rqm}
  \item $a$ or $b$ is weakly regular on $A/(x,y)$, and
  \item $a$ and $b$ are not both units.
  \end{rqm}
  Then the $A$-modules $G_a$ and $H_b$ are not isomorphic.
\end{thm}

\begin{proof}
  From the assumption (1) and \thmref{End-A} we have
  \mbox{$\Hom_A(G_a,G_a) \cong A$} or \mbox{$\Hom_A(H_b,H_b) \cong
    A$}. Thus, to prove the result it suffices to argue that
  $\Hom_A(G_a,H_b)$ is not isomorphic to $A$. From the assumption (1)
  and \thmref{Hom-H-G}(b) we have an isomorphism $\Hom_A(G_a,H_b)
  \cong H_{ab}$. It follows from the assumption (2) that $ab$ is not a
  unit, and hence $H_{ab}$ is not free by \corref{free--non-free}(b).
\end{proof}

\begin{thm}
  \label{thm:iso-GG}
  Let $x,y \in A$ be a regular exact pair of zero divisors.
  Let $a \in A$ be weakly regular on the $A$-module $A/(x,y)$, and let
  $b \in A$ be a non-unit. Then
  \begin{prt}
  \item The $A$-modules $G_a$ and $G_{ab}$ are not isomorphic.  
  \item The $A$-modules $H_a$ and $H_{ab}$ are not isomorphic.  
  \end{prt}
\end{thm}

\begin{proof}
  We will only prove part (a), as the proof of (b) is similar.

  It follows from \thmref{End-A} that \mbox{$\Hom_A(G_a,G_a) \cong
    A$}.  Thus, to prove the result it suffices to argue that
  $\Hom_A(G_{ab},G_a)$ is not isomorphic to $A$.  However, from
  \thmref{Hom-G-ab-a}(a) we get that \mbox{$\Hom_A(G_{ab},G_a) \cong
    H_b$} which is not free as $b$ is not a unit, see
  \corref{free--non-free}(b).
\end{proof}

\section{The main theorem} \label{sec:main}

This section contains the main theorem of this paper.  Although the
proof of our main result is quite short, it uses the machinery from
the previous sections and from \secref{Hom} below.  For the last claim
in the theorem, it is useful to keep in mind that \mbox{$H_1 \cong A
  \cong G_1$} by \corref{free--non-free}(a).

\begin{thm}
  \label{thm:main}
  Let \mbox{$x,y \in A$} be a regular exact pair of zero
  divisors. Let~$(b_n)_{n \geqslant 1}$ be a sequence of (not
  necessarily distinct) elements in $A$ which are regular on the
  $A$-module $A/(x,y)$.  If we define $a_n=b_1\cdots b_n$ then
  \begin{displaymath}
    G_{a_1}, G_{a_2}, G_{a_3}, G_{a_4}, \ldots, 
    H_{a_1}, H_{a_2}, H_{a_3}, H_{a_4}, \ldots
  \end{displaymath}
  is a (double) infinite family of indecomposable, non-free, totally
  reflexive $A$-modules, which are pairwise non-isomorphic.
  Furthermore, one has the following identities:
  \begin{align*}
    \Hom_A(H_{a_m},G_{a_n}) \cong \Hom_A(H_{a_n},G_{a_m}) &\cong\, G_{a_ma_n},
    \\
    \Hom_A(G_{a_n},H_{a_m}) \cong \Hom_A(G_{a_m},H_{a_n}) &\cong\, H_{a_ma_n},
    \\
    \Hom_A(G_{a_m},G_{a_n}) \cong \Hom_A(H_{a_n},H_{a_m}) &\cong 
    \left\{\!\!
      \begin{array}{ccl}
        H_{a_m/a_n} & \text{for} & m \geqslant n \\
        G_{a_n/a_m} & \text{for} & m \leqslant n
      \end{array}
    \right.\!\!.
  \end{align*}
\end{thm}

\begin{proof}
  It follows from \prpref{gamma} that $G_{a_n}$ and $H_{a_n}$ are
  totally reflexive. As $a_n$ is not a unit, $G_{a_n}$ and $H_{a_n}$
  are non-free by \corref{free--non-free}(b). Since $a_n$ is regular
  on $A/(x,y)$, it follows from \thmref{End-A} that $G_{a_n}$ and
  $H_{a_n}$ are indecomposable.

  By \thmref{iso-GH}, the modules $G_{a_m}$ and $H_{a_n}$ are not
  isomorphic. If \mbox{$m<n$} then $a_n=a_mb$ where $b=b_{m+1}\cdots
  b_n$ (which is not a unit). Thus \thmref{iso-GG} gives that
  $G_{a_m}$ and $G_{a_n}$ are not isomorphic, and furthermore that
  $H_{a_m}$ and $H_{a_n}$ are not isomorphic. Hence the
  modules in the given list are pairwise non-isomorphic.
 
  The Hom-identities are immediate from \thmref[Theorems~]{Hom-H-G} and
  \thmref[]{Hom-G-ab-a}.
\end{proof}

\enlargethispage{7ex}

\section{Computation of Hom-modules} \label{sec:Hom}

In this section, we will explicitly compute certain $\Hom$-modules.
For example, we will prove that $\Hom_A(H_b,G_a)$ is isomorphic to
$G_{ab}$. In order to carry out such computations, we need a concrete
way to represent the module of homomorphisms between two given
finitely generated $A$-modules. The representation we will use can be
found in e.g.~Greuel and Pfister \cite[Example~2.1.26]{MR2363237}; we
give a recap in \prpref[]{CompHom}.

\begin{obs}
  \label{obs:diagram}
  Consider the following commutative diagram of $A$-modules with exact
  rows and columns:
  \begin{displaymath}
    \xymatrix@R=4.1ex@C=4.1ex{
      {} & {} & 0 & 0 \\
      0 \ar[r] & K \ar[r]^-{\iota} & L \ar[r]^-{\lambda} \ar[u] & L' \ar[u] \\
      {} & {} & M \ar[r]^-{\mu} \ar[u]^-{\alpha} & M' \ar[u]_-{\alpha'} \\
      {} & {} & N \ar[r]^-{\nu} \ar[u]^-{\beta} & N' \ar[u]_-{\beta'}
    }
  \end{displaymath}
  Define submodules \mbox{$I \subseteq Q \subseteq M$} by
  \mbox{$I=\Im\beta$} and \mbox{$Q=\mu^{-1}(\Im\beta')$}. Then
  $\alpha$ maps $Q$ onto \mbox{$\Ker\lambda = \Im \iota$}, and the kernel
  of $\alpha$'s restriction $\alpha \colon Q \twoheadrightarrow
  \Im\iota$ is exactly $I$.  Consequently, there is an isomorphism of
  $A$-modules, 
  \begin{displaymath}
      Q/I \stackrel{\cong}{\longrightarrow} K,
  \end{displaymath}
  which maps $[q]_I \in Q/I$ to the unique $x \in K$ with
  $\iota(x)=\alpha(q)$. 
\end{obs}

\begin{prp}
  \label{prp:CompHom}
  Let $\rho_i$ be an \mbox{$m_i \!\times\! n_i$} matrix with entries in
  $A$, where \mbox{$i=1,2$}. Consider \mbox{$\rho_i \colon A^{n_i} \to
    A^{m_i}$} as an $A$-linear map which acts on column vectors by
  multiplication from the left. Then
  \begin{displaymath}
   Q_{\rho_1,\rho_2} = \big\{\,\psi \in \M_{m_2 \times m_1}(A) \,\big|\, 
          \psi\rho_1=\rho_2\xi \text{ for some } 
          \xi \in \M_{n_2 \times n_1}(A) \,\big\}
  \end{displaymath}
  is an $A$-submodule of $\M_{m_2 \times m_1}(A)$ which contains
  \mbox{$\rho_2\M_{n_2 \times m_1}(A)$} as a
  submodule. Furthermore, there is an isomorphism of $A$-modules,
  \begin{displaymath}
    \Hom_A(\Coker \rho_1,\Coker \rho_2) \cong 
    Q_{\rho_1,\rho_2}/(\rho_2\M_{n_2 \times m_1}(A)).
  \end{displaymath}
\end{prp}

\begin{proof}
  Set $C_i = \Coker \rho_i$. From the
  canonical exact sequences,
  \begin{displaymath}
    A^{n_i} \stackrel{\rho_i}{\longrightarrow} A^{m_i} 
    \stackrel{\pi_i}{\longrightarrow} C_i \longrightarrow 0,
  \end{displaymath}
  we get a commutative diagram of $A$-modules with exact rows and
  columns,
  \begin{displaymath}
    \xymatrix{
      {} & {} & 0 & 0 \\
      0 \ar[r] & \Hom_A(C_1,C_2) \ar[r]^-{\circ \pi_1} & \Hom_A(A^{m_1},C_2) \ar[r]^-{\circ \rho_1} \ar[u] & \Hom_A(A^{n_1},C_2) \ar[u] \\
      {} & {} & \Hom_A(A^{m_1},A^{m_2}) \ar[r]^-{\circ \rho_1} \ar[u]^-{\pi_2 \circ} & \Hom_A(A^{n_1},A^{m_2}) \ar[u]_-{\pi_2 \circ} \\
      {} & {} & \Hom_A(A^{m_1},A^{n_2}) \ar[r]^-{\circ \rho_1} \ar[u]^-{\rho_2 \circ} & \Hom_A(A^{n_1},A^{n_2}) \ar[u]_-{\rho_2 \circ}
    }
  \end{displaymath}

  We identify $\Hom_A(A^u,A^v)$ with the $A$-module $\M_{v \times
    u}(A)$ of $v \times u$ matrices with entries in $A$ (a matrix is
  viewed as an $A$-linear map which acts on column vectors by
  multiplication from the left). With this identification, the maps
  \mbox{$\circ\, \rho_1$}, respectively, \mbox{$\rho_2\, \circ$}, in
  the bottom square in the diagram above become multiplication by the
  matrix $\rho_1$ from the right, respectively, multiplication by
  $\rho_2$ from the left. 

  The proposition is now immediate from \obsref{diagram}. The isomorphism 
  \begin{displaymath}
    Q_{\rho_1,\rho_2}/(\rho_2\M_{n_2 \times m_1}(A))
    \stackrel{\cong}{\longrightarrow} \Hom_A(C_1,C_2)
  \end{displaymath}
  is given by sending $[\psi]$ to the unique homomorphism $\eta \in
  \Hom_A(C_1,C_2)$ such that $\eta \circ \pi_1 = \pi_2 \circ \psi$.
  Here $[\,\cdot\,]$ denotes coset with respect to $\rho_2\M_{n_2
    \times m_1}(A)$.
\end{proof}

\enlargethispage{2ex}

\begin{lem}
  \label{lem:generators-of-Q-2}
  Let $x,y \in A$ be a regular exact pair of zero divisors.
  Furthermore, let $a,b \in A$ be elements such that $a$ or $b$ is
  weakly regular on the $A$-module $A/(x,y)$. Then the $A$-submodule,
  \begin{displaymath}
    Q_{\eta_b,\gamma_a} = \big\{\, \psi \in \M_2(A) \,\big|\, 
    \psi\eta_b=\gamma_a\xi 
    \text{ for some } \xi \in \M_2(A) \,\big\}
  \end{displaymath}
  of $\M_2(A)$ is generated by the following five matrices,
  \begin{displaymath}
    \psi_1 = 
      \begin{pmatrix}
        0 & 1 \\
        0 & 0
      \end{pmatrix}, \
    \psi_2 = 
      \begin{pmatrix}
        0 & 0 \\
        x & b
      \end{pmatrix}, \
    \psi_3 = 
      \begin{pmatrix}
        0 & 0 \\
        0 & y
      \end{pmatrix}, \
    \psi_4 = 
      \begin{pmatrix}
        x & 0 \\
        0 & 0
      \end{pmatrix}, \
    \psi_5 = 
      \begin{pmatrix}
        a & 0 \\
        y & 0
      \end{pmatrix}.
  \end{displaymath}
\end{lem}

\begin{proof}
  First note that if we define,
  \begin{displaymath}
    \xi_1 = 
      \begin{pmatrix}
        0 & 1 \\
        0 & 0
      \end{pmatrix}, \
    \xi_2 = 
      \begin{pmatrix}
        0 & 0 \\
        0 & 0
      \end{pmatrix}, \
    \xi_3 = 
      \begin{pmatrix}
        0 & 0 \\
        0 & 0
      \end{pmatrix}, \
    \xi_4 = 
      \begin{pmatrix}
        0 & -b \\
        0 & 0
      \end{pmatrix}, \
    \xi_5 = 
      \begin{pmatrix}
        0 & 0 \\
        y & -b
      \end{pmatrix},
  \end{displaymath}
  then $\psi_i\eta_b=\gamma_a\xi_i$, and hence $\psi_i$ belongs to
  $Q_{\eta_b,\gamma_a}$.

  Next we show that $\psi_1,\ldots,\psi_5$ generate all of
  $Q_{\eta_b,\gamma_a}$. To this end, let $\psi=(b_{ij})$ be any
  matrix in $Q_{\eta_b,\gamma_a}$, that is, there exists
  $\xi=(c_{ij})$ such that $\psi\eta_b=\gamma_a\xi$; i.e.
  \begin{displaymath}
    \tag{\text{$*$}}
    \begin{pmatrix}
      b_{11}y & -b_{11}b+b_{12}x \\
      b_{21}y & -b_{21}b+b_{22}x
    \end{pmatrix}
    =
    \begin{pmatrix}
      c_{11}x+c_{21}a & c_{12}x+c_{22}a \\
      c_{21}y & c_{22}y      
    \end{pmatrix}.
  \end{displaymath}
  We must prove the existence of $f_1,\ldots,f_5$ in $A$ such that
  $\psi = \sum_{i=1}^5f_i\psi_i$, that is,
  \begin{displaymath}
    \tag{\text{$**$}}
    \begin{pmatrix}
      b_{11} & b_{12} \\
      b_{21} & b_{22}
    \end{pmatrix}
    =
    \begin{pmatrix}
      f_4x+f_5a & f_1 \\
      f_2x+f_5y & f_2b+f_3y
    \end{pmatrix}.
  \end{displaymath}
  Of course, if we define \mbox{$f_1=b_{12}$} then entry $(1,2)$ in
  $(**)$ holds. Also note that from entry $(2,1)$ in $(*)$ we get
  $(b_{21}-c_{21})y=0$, and hence $b_{21}-c_{21}=qx$ for some $q$.

  Below, we construct $f_2$ and $f_5$ such that entry $(2,1)$ in
  $(**)$ holds, that is,
  \begin{displaymath}
    \tag{\text{$\dagger$}}
    b_{21} = f_2x+f_5y.
  \end{displaymath}
  However, first we demonstrate how $f_3$ and $f_4$ can be constructed
  from $(*)$ and $(\dagger)$ such that entries $(1,1)$ and $(2,2)$ in
  $(**)$ hold:

  \emph{Existence of $f_3$}: From entry $(2,2)$ in $(*)$ we get
  $-b_{21}b+b_{22}x=c_{22}y$. Combining this with $(\dagger)$ we
  obtain $(b_{22}-f_2b)x=(c_{22}+f_5b)y$.  As the exact pair of zero
  divisors $x,y$ is regular, see \dfnref{regular}, it follows that
  $b_{22}-f_2b=f_3y$ for some $f_3$. Thus entry $(2,2)$ in $(**)$
  holds.

  \emph{Existence of $f_4$}: From entry $(1,1)$ in $(*)$, from the
  equation \mbox{$b_{21}-c_{21}=qx$} found above, and from
  $(\dagger)$, we get the following equalities,
  \begin{displaymath}
    b_{11}y = c_{11}x+c_{21}a = c_{11}x+(b_{21}-qx)a 
    = (c_{11}+f_2a-qa)x+f_5ay.
  \end{displaymath}
  Consequently, $(b_{11}-f_5a)y \in (x)$.  As the pair $x,y$ is
  regular, see \dfnref{regular}, it follows that
  \mbox{$b_{11}-f_5a=f_4x$} for some $f_4$. This shows that entry
  $(1,1)$ in $(**)$ holds.

  It remains to prove that $(\dagger)$ holds, i.e.~that $b_{21}$
  belongs to $(x,y)$. The proof is divided into two cases:

  \emph{Existence of $(\dagger)$ in the case where $a$ is weakly
    regular on $A/(x,y)$:}

  From entry $(1,1)$ in $(*)$, we have \mbox{$c_{21}a=-c_{11}x+b_{11}y
    \in (x,y)$}. As $a$ is weakly regular on $A/(x,y)$, it follows
  that \mbox{$c_{21} \in (x,y)$}. Combining this with the equation
  \mbox{$b_{21}-c_{21}=qx$} found above, it follows that \mbox{$b_{21}
    \in (x,y)$}.

  \emph{Existence of $(\dagger)$ in the case where $b$ is weakly
    regular on $A/(x,y)$:}

  From entry $(2,2)$ in $(*)$, we have \mbox{$b_{21}b=b_{22}x-c_{22}y
    \in (x,y)$}. As $b$ is weakly regular on $A/(x,y)$, it follows
  that \mbox{$b_{21} \in (x,y)$} as desired.
\end{proof}

\begin{obs}
  \label{obs:opposite pair}
  Of course, the modules $G_a$ and $H_a$ do not only depend on the
  ring element $a$, but also on the exact pair of zero divisors $x,y$.
  A more precise notation would therefore be $G_a^{x,y}$ and
  $H_a^{x,y}$. 

  Note that if $x,y$ is a (regular) exact pair of zero divisors, then
  so is $y,x$. By \dfnref{GH} and \obsref{unit}, it follows that
  $G_a^{y,x} \cong H_a^{x,y}$ and $H_a^{y,x} \cong G_a^{x,y}$.
\end{obs}


\begin{thm}
  \label{thm:Hom-H-G}
  Let $x,y \in A$ be a regular exact pair of zero divisors.
  Furthermore, let $a,b \in A$ be elements such that $a$ or $b$ is
  weakly regular on the $A$-module $A/(x,y)$. Then there are
  isomorphisms of $A$-modules,
  \begin{prt}
  \item $\Hom_A(H_a,G_b) \cong \Hom_A(H_b,G_a) \cong G_{ab}$.
  \item $\Hom_A(G_a,H_b) \cong \Hom_A(G_b,H_a) \cong H_{ab}$.
  \end{prt}
\end{thm}

\begin{proof}
  First note that by applying part (a) to the regular exact pair of
  zero divisors $y,x$ (instead of $x,y$), then part (b) follows from
  \obsref{opposite pair}.

  In part (a), the first isomorphism is by \corref{Ext}(a). To prove
  the second isomorphism, we first apply \prpref{CompHom} to get the
  representation,
  \begin{displaymath}
    \Hom_A(H_b,G_a) \cong Q_{\eta_b,\gamma_a}/(\gamma_aM_2(A)).
  \end{displaymath}
  Using this identification, we can consider the $A$-linear map,
  \begin{displaymath}
    A^2 \stackrel{\pi}{\longrightarrow} \Hom_A(H_b,G_a) \quad , \quad
    \begin{pmatrix}
      s \\ t
    \end{pmatrix}
    \longmapsto s[\psi_1]+t[\psi_2].
  \end{displaymath}
  Here $[\,\cdot\,]$ denotes coset with respect to the submodule
  $\gamma_aM_2(A)$, and $\psi_i$ are the matrices introduced in
  \lemref{generators-of-Q-2}.  Recall that $G_{ab} = \Coker
  \gamma_{ab}$; thus to prove the theorem, it suffices to show
  exactness of the sequence,
  \begin{displaymath}
    A^2 \stackrel{\gamma_{ab}}{\longrightarrow} 
    A^2 \stackrel{\pi}{\longrightarrow} \Hom_A(H_b,G_a)
    \longrightarrow 0.
  \end{displaymath}

  To show that $\pi$ is surjective, it suffices by
  \lemref{generators-of-Q-2} to see that each $[\psi_i]$ is in the
  image of $\pi$. However, this is clear since $[\psi_3]=-a[\psi_1]$
  and $[\psi_4]=[\psi_5]=[0]$.

  To see that $\pi\gamma_{ab}=0$, we must show that if
  $(s,t)=\gamma_{ab}(u,v)$, where $(u,v) \in A^2$, then
  $s\psi_1+t\psi_2$ is in $\gamma_aM_2(A)$. But as
  \begin{displaymath}
    \tag{\text{$\dagger$}}
    \begin{pmatrix}
      s \\ t
    \end{pmatrix}
    =
    \gamma_{ab}
    \begin{pmatrix}
      u \\ v
    \end{pmatrix} =
    \begin{pmatrix}
      x & ab \\ 
      0 & y
    \end{pmatrix}
    \begin{pmatrix}
      u \\ v
    \end{pmatrix}
    =
    \begin{pmatrix}
      ux+vab \\ 
      vy
    \end{pmatrix},
  \end{displaymath}
  the desired conclusion follows since,
  \begin{displaymath}
    s\psi_1+t\psi_2 = (ux+vab)
    \begin{pmatrix}
      0 & 1 \\ 
      0 & 0
    \end{pmatrix}
    + vy
    \begin{pmatrix}
      0 & 0 \\ 
      x & b
    \end{pmatrix} =
    \begin{pmatrix}
      0 & ux+vab \\ 
      0 & vyb
    \end{pmatrix} =
    \gamma_a
    \begin{pmatrix}
      0 & u \\ 
      0 & vb
    \end{pmatrix}.
  \end{displaymath}

  It remains to prove that $\Ker \pi \subseteq \Im \gamma_{ab}$. To
  this end, assume that $(s,t) \in A^2$ satisfies $\pi(s,t)=0$. This
  means that $s\psi_1+t\psi_2$ belongs to $\gamma_aM_2(A)$, that is,
  there exists $\xi = (c_{ij})$ such that $s\psi_1+t\psi_2 =
  \gamma_a\xi$; equivalently,
  \begin{displaymath}
    \tag{\text{$\ddagger$}}
    \begin{pmatrix}
      0 & s \\
      tx & tb
    \end{pmatrix}
    =
    \begin{pmatrix}
      c_{11}x+c_{21}a & c_{12}x+c_{22}a \\
      c_{21}y & c_{22}y
    \end{pmatrix}.
  \end{displaymath}

  From entry $(2,1)$ in $(\ddagger)$ we get that \mbox{$tx=c_{21}y$}.
  Since the exact pair of zero divisors $x,y$ is regular, see
  \dfnref{regular}, it follows that $t=vy$ for some $v$.

  From entry $(2,2)$ in $(\ddagger)$ we get that \mbox{$tb=c_{22}y$}.
  Combining this with \mbox{$t=vy$} one gets \mbox{$(vb-c_{22})y=0$}.
  Consequently, \mbox{$vb-c_{22}=px$} for some $p$. Now, inserting
  \mbox{$c_{22}=vb-px$} in the equation \mbox{$s= c_{12}x+c_{22}a$}
  coming from entry $(1,2)$ in $(\ddagger)$, we get that
  \mbox{$s=(c_{12}-pa)x+vab = ux+vab$}, where \mbox{$u=c_{12}-pa$}.

  Since $s=ux+vab$ and $t=vy$, we see from $(\dagger)$ that
  $(s,t)=\gamma_{ab}(u,v) \in \Im \gamma_{ab}$.
\end{proof}



It is also desirable to know what the module \mbox{$\Hom_A(G_u,G_v) \cong
  \Hom_A(H_v,H_u)$} (see \corref{Ext}(c) for this isomorphism) looks
like for every combination of $u$~and~$v$.  Since the author was not
able to figure this out, we restrict ourselves to the case where $u$
is in the ideal generated by $v$, or vice versa, see
\thmref{Hom-G-ab-a}.

\begin{lem}
  \label{lem:generators-of-Q}
  Let $x,y \in A$ be a regular exact pair of zero divisors.
  Furthermore, let $a \in A$ be weakly regular on the $A$-module
  $A/(x,y)$, and let $b \in A$ be any element. Then the $A$-submodule,
  \begin{displaymath}
    Q_{\gamma_{ab},\gamma_a} = \big\{\, \psi \in \M_2(A) \,\big|\, 
    \psi\gamma_{ab}=\gamma_a\xi 
    \text{ for some } \xi \in \M_2(A) \,\big\}
  \end{displaymath}
  of $\M_2(A)$ is generated by the following five matrices,
  \begin{displaymath}
    \psi_1 = 
      \begin{pmatrix}
        0 & 0 \\
        0 & x
      \end{pmatrix}, \
    \psi_2 = 
      \begin{pmatrix}
        1 & 0 \\
        0 & b
      \end{pmatrix}, \
    \psi_3 = 
      \begin{pmatrix}
        0 & x \\
        0 & 0
      \end{pmatrix}, \
    \psi_4 = 
      \begin{pmatrix}
        a & 0 \\
        y & 0
      \end{pmatrix}, \
    \psi_5 = 
      \begin{pmatrix}
        0 & a \\
        0 & y
      \end{pmatrix}.
  \end{displaymath}
\end{lem}

\begin{proof}
  First note that if we define,
  \begin{displaymath}
    \xi_1 = 
      \begin{pmatrix}
        0 & 0 \\
        0 & 0
      \end{pmatrix}, \
    \xi_2 = 
      \begin{pmatrix}
        1 & 0 \\
        0 & b
      \end{pmatrix}, \
    \xi_3 = 
      \begin{pmatrix}
        0 & 0 \\
        0 & 0
      \end{pmatrix}, \
    \xi_4 = 
      \begin{pmatrix}
        a & 0 \\
        0 & ab
      \end{pmatrix}, \
    \xi_5 = 
      \begin{pmatrix}
        0 & 0 \\
        0 & y
      \end{pmatrix},
  \end{displaymath}
  then $\psi_i\gamma_{ab}=\gamma_a\xi_i$, and hence $\psi_i$ belongs to
  $Q_{\gamma_{ab},\gamma_a}$.

  Next we show that $\psi_1,\ldots,\psi_5$ generate all of
  $Q_{\gamma_{ab},\gamma_a}$. To this end, let $\psi=(b_{ij})$ be any
  matrix in $Q_{\gamma_{ab},\gamma_a}$, that is, there exists
  $\xi=(c_{ij})$ such that $\psi\gamma_{ab}=\gamma_a\xi$; i.e.
  \begin{displaymath}
    \tag{\text{$*$}}
    \begin{pmatrix}
      b_{11}x & b_{11}ab+b_{12}y \\
      b_{21}x & b_{21}ab+b_{22}y
    \end{pmatrix}
    =
    \begin{pmatrix}
      c_{11}x+c_{21}a & c_{12}x+c_{22}a \\
      c_{21}y & c_{22}y
    \end{pmatrix}.
  \end{displaymath}
  We must prove the existence of $f_1,\ldots,f_5$ in $A$ such that
  $\psi = \sum_{i=1}^5f_i\psi_i$, that is,
  \begin{displaymath}
    \tag{\text{$**$}}
    \begin{pmatrix}
      b_{11} & b_{12} \\
      b_{21} & b_{22}
    \end{pmatrix}
    =
    \begin{pmatrix}
      f_2+f_4a & f_3x+f_5a \\
      f_4y & f_1x+f_2b+f_5y
    \end{pmatrix}.
  \end{displaymath}

  From entry $(2,1)$ in $(*)$ we get that
  \mbox{$b_{21}x=c_{21}y$}.  By assumption, the exact pair of zero
  divisors $x,y$ is regular, see \dfnref{regular}, so it follows that
  $b_{21}=f_4y$ for some $f_4$, which gives entry $(2,1)$ in $(**)$.

  From entry $(1,2)$ in $(*)$ we get that
  $(b_{11}b-c_{22})a=c_{12}x-b_{12}y \in (x,y)$. As $a$ is weakly
  regular on $A/(x,y)$, there exist $p$ and $f_5$ such that
  \begin{displaymath}
    \tag{\text{$\dagger$}}
    b_{11}b-c_{22} = px-f_5y.
  \end{displaymath}

  From entry $(2,2)$ in $(*)$ we get
  \mbox{$b_{21}ab+b_{22}y=c_{22}y$}.  Since \mbox{$b_{21}=f_4y$}, it
  follows that \mbox{$(f_4ab+b_{22}-c_{22})y=0$}. Consequently, there
  exists $q$ such that
  \begin{displaymath}
    \tag{\text{$\ddagger$}}
    f_4ab+b_{22}-c_{22}=qx.
  \end{displaymath}
  
  Now define \mbox{$f_1=q-p$} and \mbox{$f_2=b_{11}-f_4a$}. Clearly,
  entry $(1,1)$ in $(**)$ holds. Furthermore, subtracting $(\dagger)$
  from $(\ddagger)$ and rearranging terms gives,
  \begin{displaymath}
    b_{22} = (q-p)x+(b_{11}-f_4a)b+f_5y = f_1x+f_2b+f_5y,
  \end{displaymath}
  which shows that also entry $(2,2)$ in $(**)$ holds.

  Finally, \mbox{$c_{12}x-b_{12}y = (b_{11}b-c_{22})a = (px-f_5y)a$}.
  The first equality is from entry $(1,2)$ in $(*)$, and the second is
  by $(\dagger)$.  Consequently,
  \mbox{$(b_{12}-f_5a)y=(c_{12}-pa)x$}. Since the exact pair of
  zero divisors $x,y$ is regular, see \dfnref{regular}, it follows
  that $b_{12}-f_5a=f_3x$ for some $f_3$.  Hence entry $(1,2)$ in
  $(**)$ holds. 
\end{proof}

\begin{thm}
  \label{thm:Hom-G-ab-a}
  Let $x,y \in A$ be a regular exact pair of zero divisors.
  Furthermore, let $a \in A$ be weakly regular on the $A$-module
  $A/(x,y)$, and let $b \in A$ be any element. Then there are
  isomorphisms of $A$-modules,
  \begin{prt}
  \item $\Hom_A(H_a,H_{ab}) \cong \Hom_A(G_{ab},G_a) \cong H_b$.
  \item $\Hom_A(G_a,G_{ab}) \cong \Hom_A(H_{ab},H_a) \cong G_b$.
  \end{prt}
\end{thm}

\begin{proof}
  First note that by applying part (a) to the regular exact pair of
  zero divisors $y,x$ (instead of $x,y$), then part (b) follows from
  \obsref{opposite pair}.
  
  In part (a), the first isomorphism is by \corref{Ext}(c). To prove
  the second isomorphism, we apply \prpref{CompHom} to get the
  representation,
  \begin{displaymath}
    \Hom_A(G_{ab},G_a) \cong Q_{\gamma_{ab},\gamma_a}/(\gamma_aM_2(A)).
  \end{displaymath}
  Using this identification, we can consider the $A$-linear map,
  \begin{displaymath}
    A^2 \stackrel{\pi}{\longrightarrow} \Hom_A(G_{ab},G_a) \quad , \quad
    \begin{pmatrix}
      s \\ t
    \end{pmatrix}
    \longmapsto s[\psi_1]+t[\psi_2].
  \end{displaymath}
  Here $[\,\cdot\,]$ denotes coset with respect to the submodule
  $\gamma_aM_2(A)$, and $\psi_i$ are the matrices introduced in
  \lemref{generators-of-Q}. Recall that $H_b = \Coker \eta_b$, see
  \dfnref{GH}; thus to prove the result, it suffices to show
  exactness of the sequence,
  \begin{displaymath}
    A^2 \stackrel{\eta_b}{\longrightarrow} 
    A^2 \stackrel{\pi}{\longrightarrow} \Hom_A(G_{ab},G_a)
    \longrightarrow 0.
  \end{displaymath}

  To see that $\pi$ is surjective, it suffices by
  \lemref{generators-of-Q} to see that each $[\psi_i]$ belongs to
  $\Im\pi$.  However, this is clear since
  $[\psi_3]=[\psi_4]=[\psi_5]=[0]$.

  To see that $\pi\eta_b=0$, we must show that if $(s,t)=\eta_b(u,v)$,
  where $(u,v) \in A^2$, then $s\psi_1+t\psi_2$ is in
  $\gamma_aM_2(A)$. But as
  \begin{displaymath}
    \tag{\text{$\dagger$}}
    \begin{pmatrix}
      s \\ t
    \end{pmatrix}
    =
    \eta_b
    \begin{pmatrix}
      u \\ v
    \end{pmatrix} =
    \begin{pmatrix}
      y & -b \\ 
      0 & x
    \end{pmatrix}
    \begin{pmatrix}
      u \\ v
    \end{pmatrix}
    =
    \begin{pmatrix}
      uy-vb \\ 
      vx
    \end{pmatrix},
  \end{displaymath}
  it follows that
  \begin{displaymath}
    s\psi_1+t\psi_2 = (uy-vb)
    \begin{pmatrix}
      0 & 0 \\ 
      0 & x
    \end{pmatrix}
    + vx
    \begin{pmatrix}
      1 & 0 \\ 
      0 & b
    \end{pmatrix} =
    \begin{pmatrix}
      vx & 0 \\ 
      0 & 0
    \end{pmatrix} =
    \gamma_a
    \begin{pmatrix}
      v & 0 \\ 
      0 & 0
    \end{pmatrix},
  \end{displaymath}
  which belongs to $\gamma_aM_2(A)$. 

  It remains to prove that \mbox{$\Ker \pi \subseteq \Im \eta_b$}. To
  this end, assume that \mbox{$(s,t) \in A^2$} satisfies $\pi(s,t)=0$.
  This means that $s\psi_1+t\psi_2$ belongs to $\gamma_aM_2(A)$, that
  is, there exists $\xi = (c_{ij})$ such that
  $s\psi_1+t\psi_2=\gamma_a\xi$; equivalently,
  \begin{displaymath}
    \tag{\text{$\ddagger$}}
    \begin{pmatrix}
      t & 0 \\
      0 & sx+tb
    \end{pmatrix}
    =
    \begin{pmatrix}
      c_{11}x+c_{21}a & c_{12}x+c_{22}a \\
      c_{21}y & c_{22}y
    \end{pmatrix}.
  \end{displaymath}

  From entry $(2,1)$ in $(\ddagger)$ we get $c_{21}y=0$, and
  consequently \mbox{$c_{21}=px$} for some $p$.  Combining this with
  entry $(1,1)$ in $(\ddagger)$, we get $t=(c_{11}+pa)x = vx$
  for $v=c_{11}+pa$.

  Since \mbox{$t=vx$} we get from entry $(2,2)$ in $(\ddagger)$ that
  $(s+vb)x=c_{22}y$. As the exact pair of zero divisors $x,y$ is
  regular, see \dfnref{regular}, it follows that \mbox{$s+vb=uy$} for
  some $u$, that is, $s=uy-vb$.

  As $s=uy-vb$ and $t=vx$, we see from $(\dagger)$ that
  $(s,t)=\eta_b(u,v) \in \Im \eta_b$.
\end{proof}

\def\cprime{$'$}
  \newcommand{\arxiv}[2][AC]{\mbox{\href{http://arxiv.org/abs/#2}{\sf arXiv:#2
  [math.#1]}}}
  \newcommand{\oldarxiv}[2][AC]{\mbox{\href{http://arxiv.org/abs/math/#2}{\sf
  arXiv:math/#2
  [math.#1]}}}\providecommand{\MR}[1]{\mbox{\href{http://www.ams.org/mathscine%
t-getitem?mr=#1}{#1}}}
  \renewcommand{\MR}[1]{\mbox{\href{http://www.ams.org/mathscinet-getitem?mr=#%
1}{#1}}}
\providecommand{\bysame}{\leavevmode\hbox to3em{\hrulefill}\thinspace}
\providecommand{\MR}{\relax\ifhmode\unskip\space\fi MR }
\providecommand{\MRhref}[2]{%
  \href{http://www.ams.org/mathscinet-getitem?mr=#1}{#2}
}
\providecommand{\href}[2]{#2}


\begin{thebibliography}{10}

\bibitem{MAs67}
Maurice Auslander, \emph{Anneaux de {G}orenstein, et torsion en alg\`ebre
  commutative}, Secr\'etariat math\'ematique, Paris, 1967, S\'eminaire
  d'Alg\`ebre Commutative dirig\'e par Pierre Samuel, 1966/67. Texte
  r\'edig\'e, d'apr\`es des expos\'es de Maurice Auslander, par Marquerite
  Mangeney, Christian Peskine et Lucien Szpiro. \'Ecole Normale Sup\'erieure de
  Jeunes Filles. Available from \mbox{\sffamily http://www.numdam.org}.
  \MR{MR0225844}

\bibitem{AGP-97}
Luchezar~L. Avramov, Vesselin~N. Gasharov, and Irena~V. Peeva, \emph{Complete
  intersection dimension}, Inst. Hautes \'Etudes Sci. Publ. Math. (1997),
  no.~86, 67--114 (1998). \MR{MR1608565}

\bibitem{LLASBI07}
Luchezar~L. Avramov and Srikanth~B. Iyengar, \emph{Constructing modules with
  prescribed cohomological support}, Illinois J. Math. \textbf{51} (2007),
  no.~1, 1--20. \MR{MR2346182}

\bibitem{LLAAMr02}
Luchezar~L. Avramov and Alex Martsinkovsky, \emph{Absolute, relative, and
  {T}ate cohomology of modules of finite {G}orenstein dimension}, Proc. London
  Math. Soc. (3) \textbf{85} (2002), no.~2, 393--440. \MR{MR1912056}

\bibitem{BruHer}
Winfried Bruns and J{\"u}rgen Herzog, \emph{Cohen-{M}acaulay rings}, Cambridge
  Studies in Advanced Mathematics, vol.~39, Cambridge University Press,
  Cambridge, 1993. \MR{MR1251956}

\bibitem{CPST-08}
Lars~Winther Christensen, Greg Piepmeyer, Janet Striuli, and Ryo Takahashi,
  \emph{Finite {G}orenstein representation type implies simple singularity},
  Adv. Math. \textbf{218} (2008), no.~4, 1012--1026. \MR{MR2419377}

\bibitem{MR2363237}
Gert-Martin Greuel and Gerhard Pfister, \emph{A {\bf {s}ingular} introduction
  to commutative algebra}, extended ed., Springer, Berlin, 2008, With
  contributions by Olaf Bachmann, Christoph Lossen and Hans Sch{\"o}nemann,
  With 1 CD-ROM (Windows, Macintosh and UNIX). \MR{MR2363237 (2008j:13001)}

\bibitem{RCH93}
Raymond~C. Heitmann, \emph{Characterization of completions of unique
  factorization domains}, Trans. Amer. Math. Soc. \textbf{337} (1993), no.~1,
  379--387. \MR{MR1102888}

\bibitem{HS}
In{\^e}s~B. Henriques and Liana~M. {\c{S}}ega, \emph{Free resolutions over
  short {G}orenstein local rings}, preprint (2009), \arxiv{0904.3510v2}.

\bibitem{Mat}
Hideyuki Matsumura, \emph{Commutative ring theory}, second ed., Cambridge
  Studies in Advanced Mathematics, vol.~8, Cambridge University Press,
  Cambridge, 1989, Translated from the Japanese by M. Reid. \MR{MR1011461}

\bibitem{RTk07a}
Ryo Takahashi, \emph{An uncountably infinite number of indecomposable totally
  reflexive modules}, Nagoya Math. J. \textbf{187} (2007), 35--48.
  \MR{MR2354554}

\bibitem{Wei}
Charles~A. Weibel, \emph{An introduction to homological algebra}, Cambridge
  Studies in Advanced Mathematics, vol.~38, Cambridge University Press,
  Cambridge, 1994. \MR{MR1269324}

\end{thebibliography}

\end{document}